\documentclass[12pt,leqno]{article}
\usepackage[shortlabels]{enumitem}
\usepackage[utf8]{inputenc} 
\usepackage[T1]{fontenc}
\usepackage{preamble}
\usepackage{tikz-cd}

\title{Uniform weak RC-positivity and rational connectedness 
}
\author{Kuang-Ru Wu}

\usepackage{fullpage}
\usepackage{multicol}

\begin{document}

\date{}

\parskip=6pt

\maketitle

\begin{abstract}
In this paper, we show that if the holomorphic tangent bundle $TX$ of a compact K\"ahler manifold $X$ is uniformly weakly RC-positive, then $X$ is projective and rationally connected. This result is previously established by Xiaokui Yang under the stronger assumption that $TX$ is  uniformly RC-positive.  

The result we obtain is, in fact, more general. If a holomorphic vector bundle $E$ is uniformly weakly RC-positive, then $E$ admits a Hermitian metric whose mean curvature is positive. A quasi-positive version is also proved in this paper.
\end{abstract}

\section{Introduction}
The notion of positivity has been an important topic in differential geometry. Their characterization through the positivity notion in algebraic geometry is of particular interest. Since once the characterization is established, one can then transport tools from one area to the other. The positivity notion we study in this paper is RC-positivity whose algebro-geometric counterpart is rational connectedness. RC-positivity is first introduced by Xiaokui Yang in \cite{YangCamb} to solve a conjecture of Yau on projectivity and rational connectedness of a compact K\"ahler manifold with positive holomorphic sectional curvature. Moreover, variants of RC-positivity such as uniform RC-positivity and weak RC-positivity are also introduced and studied by Yang in \cite{YangCamb, YangForum}. 

We first recall the definitions of RC-positivity and its variants. Let $E$ be a holomorphic vector bundle of rank $r$ over a compact complex manifold $X$ of dimension $n$. Given a Hermitian metric $H$ on $E$, we denote the Chern curvature of $H$ by $\Theta^H$. For $u\in E_t$ and $v\in T_tX$ with $t\in X$, we define 
\begin{equation}\label{curvature} H(\Theta^H u,u)(v,\bar{v}):=\sum_{j,k}H(\Theta^H_{j\bar{k}}u,u  )v_j\bar{v}_k  
\end{equation} if we locally write the curvature $\Theta^H=\sum_{j,k}\Theta^H_{j\bar{k}}dt_j\wedge d\bar{t}_k$ and $v=\sum_j v_j \partial/\partial t_j$ (it is clear that (\ref{curvature}) is independent of the choice of coordinates).
\begin{definition}\label{def 1}
    \begin{enumerate}
        \item The Hermitian metric $H$ is called RC-positive if for any $t\in X$ and any nonzero $u\in E_t$, there exists a nonzero $v\in T_t X$ such that $H(\Theta^H u,u)(v,\bar{v})>0$.
        \item The Hermitian metric $H$ is called uniformly RC-positive if for any $t\in X$, there exists a nonzero $v\in T_tX$ such that $H(\Theta^H u,u)(v,\bar{v})>0$ for any nonzero $u\in E_t$.
    \end{enumerate} 
\end{definition}
A bundle $E$ is called (uniformly) RC-positive if it admits a (uniformly) RC-positive Hermitian metric. On the other hand, (uniform) weak RC-positivity is defined on the line bundle $O_{P(E^*)}(1)$ over the projectivized bundle $P(E^*)$ where $E^*$ is the dual bundle of $E$.
\begin{definition}\label{def 2}
    \begin{enumerate}
        \item A bundle $E$ is called weakly RC-positive if there is a metric $h$ on $O_{P(E^*)}(1)$ whose curvature $\Theta$ is positive on every fiber $P(E^*_t)$ and $\Theta$ has at least $r$ positive eigenvalues at every point in $P(E^*)$.
        \item A bundle $E$ is called uniformly weakly RC-positive if there is a metric $h$ on $O_{P(E^*)}(1)$ whose curvature $\Theta$ is positive on every fiber $P(E^*_t)$. Moreover, for any $t\in X$, there is a nonzero $v\in T_tX$ such that $\Theta(\Tilde{v},\bar{\Tilde{v}})|_{(t,[\zeta])}>0$ for any lift $\Tilde{v}$ of $v$ to $T_{(t,[\zeta])}P(E^*)$.           
    \end{enumerate}
\end{definition}

By definition and standard computations for Hermitian bundles, the following relation holds (for motivation and more details about the variants of RC-positivity, see \cite{wu2025uniform}).
\[\begin{tikzcd}
\text{uniform RC-positivity} \arrow{r} \arrow[swap]{d} & \text{RC-positivity} \arrow{d} \\
\text{uniform weak RC-positivity} \arrow{r} & \text{weak RC-positivity}
\end{tikzcd}
\]
It is conjectured by Yang \cite[Question 7.11]{YangCamb} that weak RC-positivity implies RC-positivity. Motivated by this conjecture, we introduce the concept of uniform weak RC-positivity in \cite{wu2025uniform} and ask 
\begin{question}\label{Q}
  If $E$ is uniformly weakly RC-positive, then is $E$  uniformly RC-positive? 
\end{question}
In \cite[Theorem 3]{wu2025uniform}, we are able to show that if $E$ is uniformly weakly RC-positive over a compact K\"ahler manifold, then $S^kE\otimes \det E$ is uniformly RC-positive for any $k\geq 0$ and $S^kE$ is uniformly RC-positive for $k$ large. The first result of the paper is
\begin{theorem}\label{cor}
If $E$ is uniformly weakly RC-positive over a compact K\"ahler manifold, then $E$ is RC-positive. 
\end{theorem}
Note that the conclusion is about $E$ itself - there is no high symmetric power or tensoring with $\det E$. This result
provides another evidence in favor of the above Question. Both Yang's conjecture \cite[Question 7.11]{YangCamb} and our uniform version of Yang's conjecture are reminiscent of a conjecture of Griffiths \cite{Griff69}, which says that if $E$ is ample, then $E$ is Griffiths positive. For the developments of the Griffiths conjecture, see \cite{Umemura,CampanaFlenner,Berndtsson09,MourouganeTaka,positivityandvanishingthmliu,liu2014curvatures,FengLiuWan,demailly2020hermitianyangmills,pingali2021note,finski2020monge,wu_2022,wupositivelyII,wuIII,lempert2024two,murakami2025analytic,wu2025mean}.

The second result of this paper is the following theorem, which shows how the positivity of the holomorphic tangent bundle $TX$ affects $X$.
\begin{theorem}\label{thm1}
    If the holomorphic tangent bundle $TX$ of a compact K\"ahler manifold $X$ is uniformly weakly RC-positive, then $X$ is projective and rationally connected. 
\end{theorem}
Theorem \ref{thm1} is previously proved by Yang \cite[Theorem 1.3]{YangForum} under the stronger assumption that $TX$ is  uniformly RC-positive. Not only does Theorem \ref{thm1} improve Yang's result, it also provides evidence in favor of the conjecture: uniform weak RC-positivity implies uniform RC-positivity.

The proofs of Theorem \ref{cor} and Theorem \ref{thm1} are inspired by an observation in \cite{wu2025uniform}, which we now describe. Let $p:\mathcal{X}^{n+m}\to Y^m$ be a proper holomorphic surjection between two complex manifolds. We assume  $\mathcal{X}$ is K\"ahler and the differential $dp$ is surjective at every point. We denote the fibers $p^{-1}(t)$ by $\mathcal{X}_t$ for $t\in Y$. Let $(L,h)$ be a Hermitian line bundle over $\mathcal{X}$. Let $$V_t=H^0(\mathcal{X}_t, L|_{\mathcal{X}_t}\otimes K_{\mathcal{X}_t}).$$ We assume that $\dim V_t$ is independent of $t\in Y$. Therefore, the direct image of the sheaf of sections of $L\otimes K_{\mathcal{X}/Y}$ is locally free by Grauert's direct image theorem, where $K_{\mathcal{X}/Y}$ is the relative canonical bundle. We denote by $V$ the associated vector bundle over $Y$. There is an $L^2$-Hermitian metric $H$ on $V$ defined as follows. For $u$ in $V_t$ with $t\in Y$, 
\begin{equation}\label{L2}
H(u,u):=\int_{\mathcal{X}_t}h(u,u).
\end{equation}
Here, we extend the metric $h$ to act on sections $u$ of $L|_{\mathcal{X}_t}\otimes K_{\mathcal{X}_t}$ so that $h(u,u)$ is an $(n,n)$-form on $\mathcal{X}_t$. In terms of local coordinates, if $u=u'\otimes e$ with $u'$ an $(n,0)$-form and $e$ a frame of $L|_{\mathcal{X}_t}$, then $h(u,u)=c_n u' \wedge \overline{u'} h(e,e)$ where $c_n=i^{n^2}$.

In \cite[Lemma 6]{wu2025uniform}, we show that if the curvature $\Theta$ of $h$ is positive on every fiber $\mathcal{X}_t$, then the following are equivalent.
\begin{enumerate}
    \item The curvature $\Theta$ has at least $n+1$ positive eigenvalues at every point in $\mathcal{X}$ (namely, weakly RC-positive).
    \item There exists a positive $\beta\in C^\infty (\mathcal{X}, p^*(\wedge^{1,1}T^*Y) )   $ such that $\Theta^{n+1}\wedge \beta^{m-1}>0$.
\end{enumerate}
The point of rephrasing weak RC-positivity through $\Theta^{n+1}\wedge \beta^{m-1}>0$ is that the expression $\Theta^{n+1}\wedge \beta^{m-1}>0$ almost fits into the framework of \cite{wu2025mean}, especially the following theorem.
\begin{theorem}{\cite[Theorem 2]{wu2025mean}}\label{thm mean}
    If the curvature $\Theta$ of $h$ is positive on every fiber $\mathcal{X}_t$ and $\Theta^{n+1}\wedge p^*\alpha^{m-1}>0$ for some Hermitian metric $\alpha$ on $Y$, then the Hermitian bundle $(V,H)$ has positive mean curvature $\Lambda_\alpha \Theta^H>0$.
\end{theorem}
Here, $\Lambda_\alpha$ is the trace with respect to the Hermitian metric $\alpha$, so $\Lambda_\alpha\Theta^H$ is $\text{End}V$-valued. We follow Kobayashi in \cite[Section 3.1]{reprintDiffofcomplexbundles} and call $\Lambda_\alpha\Theta^H$ the mean curvature of $(V,H)$ with respect to $\alpha$. Our original motivation for studying the expression $\Theta^{n+1}\wedge p^*\alpha^{m-1}$ comes from the Wess--Zumino--Witten equation in the space of K\"ahler potentials (see \cite{wu23,wu2024potential,finski2024WZW,finski2024lower}). However, in this paper, we uncover a new connection with RC-positivity (see Lemma \ref{lem} below).

In order to use Theorem \ref{thm mean}, we have to make sure the $\beta$ in $\Theta^{n+1}\wedge \beta^{m-1}$ can be written as $\beta=p^*\alpha$ for some Hermitian metric $\alpha$ on $Y$. We prove the following key lemma showing that this is doable in the case of uniform weak RC-positivity. 
\begin{lemma}\label{lem}
    If the curvature $\Theta$ of $h$ is uniformly weakly RC-positive in the sense of Definition \ref{def 2}, then there exists a Hermitian metric $\alpha$ on $Y$ such that $\Theta^{n+1}\wedge p^*\alpha^{m-1}>0$.
\end{lemma}
 The converse of Lemma \ref{lem} seems to be wrong, but we do not have an example yet. If we apply the fibration $p:\mathcal{X}\to Y$ to the special case $P(E^*)\to X$, and choose the line bundle $L$ suitably, then using Lemma \ref{lem}, Theorem \ref{thm mean}, and \cite[Theorem 1.4]{LiZhangZhang}, we obtain
\begin{theorem}\label{thm general}
    If $E$ is uniformly weakly RC-positive over a compact K\"ahler manifold $X$, then $E$ admits a Hermitian metric $H$ whose mean curvature $\Lambda_\alpha \Theta^{H}$ is positive with respect to some Hermitian metric $\alpha$ on $X$.
\end{theorem}
The assumption $X$ being K\"ahler can probably be relaxed because it is only used to guarantee that $P(E^*)$ is K\"ahler.  
There is another way to prove Theorem \ref{thm general}, which uses the uniform RC-positivity of $S^kE$ for large $k$  in \cite{wu2025uniform} and a proposition in  \cite[Proposition 3.6]{LiZhangZhang} (see the Proof of Theorem \ref{thm general} for details). Once we have Theorem \ref{thm general}, we can immediately deduce Theorem \ref{cor} and  Theorem \ref{thm1}. Indeed, Theorem \ref{cor} follows from the fact that positive mean curvature implies RC-positivity (\cite[Theorem 3.6]{YangCamb}). Theorem \ref{thm1} follows from a result of Yang \cite[Corollary 1.5]{YangCamb}, which basically says that mean curvature positivity implies rational connectedness.

Next, we turn to the quasi-positive case. One can easily define the quasi-positive counterpart of  Definitions \ref{def 1}  and \ref{def 2}. For example, a bundle $E$ is called uniformly weakly RC-quasi-positive if it is uniformly weakly RC-semipositive everywhere in $X$ and uniformly weakly RC-positive at some point in $X$, that is, there exist a point $t\in X$ and a nonzero $v\in T_tX$ such that $\Theta(\Tilde{v},\bar{\Tilde{v}})|_{(t,[\zeta])}> 0$ for any lift $\Tilde{v}$ of $v$ to $T_{(t,[\zeta])}P(E^*)$.

We consider the fibration $p:\mathcal{X}^{m+n}\to Y^m$ and the Hermitian line bundle $(L,h)\to \mathcal{X}$ introduced earlier. By carefully analyzing the proof of Theorem \ref{thm mean}, we obtain the following  quantitative version.

\begin{theorem}\label{thm quantitative}
   For a fixed $t_0\in Y$, if the curvature $\Theta$ of $h$ is positive on the fiber $\mathcal{X}_{t_0}$, and $\Theta^{n+1}\wedge p^*\alpha^{m-1}>M \Theta^{n}\wedge p^*\alpha^{m}$ for any point on $\mathcal{X}_{t_0}$ where $\alpha$ is a Hermitian metric on $Y$ and $M$ is a constant, then the mean curvature $\Lambda_\alpha \Theta^H$ of the Hermitian bundle $(V,H)$ satisfies  $\Lambda_\alpha \Theta^H>Mm/(n+1) \Id_{E}$ at $t_0$.
\end{theorem}
Theorem \ref{thm quantitative} is also true if we replace the strict inequality $``>"$ with inequality $``\geq"$. Applying Theorem \ref{thm quantitative} to the fiberation $P(E^*)\to X$ with $X$ K\"ahler, we arrive at the following result, which improves \cite[Theorem 5]{wu2025mean} to the quasi-positive case.
\begin{theorem}\label{thm quasi to full}
Assume $O_{P(E^*)}(1)$ admits a metric $h$ whose curvature $\Theta$ is positive on every fiber $P(E^*_t)$. If $\Theta^{r}\wedge p^*\alpha^{n-1}\geq  0$ on $P(E^*)$ for some Hermitian metric $\alpha$ on $X$, and $\Theta^{r}\wedge p^*\alpha^{n-1}>  0$ for any point on $P(E^*_{t_0})$ for some $t_0\in X$, then $E$ admits a Hermitian metric $H$ whose mean curvature $\Lambda_{\alpha} \Theta^H$ is positive.
\end{theorem}
In the proof of Theorem \ref{thm quasi to full}, we use Theorem \ref{thm quantitative} to get estimates on the curvature of $S^kE$, and then take $k$ large so that posivity of the curvature peaks in a neighborhood of $t_0$.
The motivation behind establishing Theorem \ref{thm quasi to full} comes from a conjecture of Yang \cite[Conjecture 1.9]{YangForum}:
\begin{conjecture}
Let $X$ be a compact K\"ahler manifold. If the holomorphic tangent bundle $TX$ admits a Hermitian metric $\omega$ that is uniformly RC-quasi-positive (or has quasi-positive holomorphic sectional curvature), then $X$ is projective and rationally connected. 
\end{conjecture}
If a quasi-positive version of Lemma \ref{lem} were available (which we have not been able to prove), then the RC-quasi-positivity part of the Conjecture above would follow from Theorem \ref{thm quasi to full}. Nevertheless, using the strategy in \cite{wu2025uniform}, we make the following progress.
\begin{theorem}\label{thm uniruled}
    Let $X$ be a compact K\"ahler manifold. If the holomorphic tangent bundle $TX$  is uniformly RC-quasi-positive (actually, uniform weak RC-quasi-positivity will do), then $K^{-1}_X$ is RC-quasi-positive.
\end{theorem}
It seems to be an open question whether  RC-quasi-positivity of $K^{-1}_X$ implies  RC-positivity of  $K^{-1}_X$. If the answer to the question is affirmative, then by \cite[Theorem 1.5]{YangCompos} $K_X$ is not pseudo-effective and then by \cite[Theorem 1.1]{ou2025characterization} $X$ is uniruled.

Finally, let us remark that if a compact  K\"ahler manifold $(X,\omega)$ has quasi-positive holomorphic sectional curvature, then $(TX,\omega)$ is uniformly RC-quasi-positive (this fact can be proved using the same argument in \cite[Theorem 5.1]{YangForum}). So, RC-positivity is in general a weaker assumption. On the other hand, under the stronger assumption (quasi-positivity of holomorphic sectional curvature), a recent progress on the Conjecture above is that 
\begin{enumerate}
    \item If $X$ is projective and admits a K\"ahler metric $\omega$ with quasi-positive holomorphic sectional curvature, then $X$ is rationally connected (\cite[Theorem 1.5]{HeierWong} and \cite{MatsumuraAJM}). 
    \item If $X$ admits a K\"ahler metric $\omega$ with quasi-positive holomorphic sectional curvature, then $X$ is projective and rationally connected  (\cite[Theorem 1.5]{zhang2023structure} and \cite[Corollary 1.3]{matsumura2025fundamental}). 
\end{enumerate}
There are also results about quasi-positive mixed curvature, see \cite{chu2025kahler,tang2026quasi}.

The paper is organized as follows. In Section \ref{sec prelim}, we collect two lemmas instrumental in the proof of Lemma \ref{lem}. In Section \ref{sec main}, we prove Lemma \ref{lem}, Theorem \ref{thm general}, and Theorem \ref{thm1}. In Section \ref{sec quasi}, we prove Theorem \ref{thm quantitative}, Theorem \ref{thm quasi to full}, and Theorem \ref{thm uniruled}. 

I am grateful to L\'aszl\'o Lempert for his critical remarks on the draft of the paper. I would like to thank National Central University for the support.

\section{Preliminary}\label{sec prelim}
In this section, we review a few formulas that will be used later. Let $p:\mathcal{X}^{m+n}\to Y^m$ be the fibration in the introduction, and $(L,h)\to \mathcal{X}$ the Hermitian line bundle. Assume that the metric $h$ on the line bundle $L$ has its curvature $\Theta$ positive on every fiber $\mathcal{X}_t$. For such an $h$, we can decompose its curvature $\Theta=\Theta_{\mathcal{H}}+\Theta_\mathcal{V}$ where $\Theta_\mathcal{H}$ is the horizontal component and  $\Theta_\mathcal{V}$ is the vertical component. The horizontal component $\Theta_\mathcal{H}$ can be viewed as an element in $C^{\infty}(\mathcal{X}, p^*(\wedge^{1,1}T^*Y))$.

To write the local expressions for $\Theta_\mathcal{H}$ and $\Theta_\mathcal{V}$, we denote by $(t_1,\ldots,t_m)$ the local coordinates in $Y$ and by $(t_1,\ldots,t_m,z_1,\ldots, z_n)$ the local coordinates in $\mathcal{X}$, and assume that $h=e^{-\phi}$ locally. We write $\phi_{i\bar{j}}=\phi_{t_i\bar{t}_j}$, $\phi_{\lambda\bar{\mu}}=\phi_{z_\lambda \bar{z}_\mu}$, etc. Since $\Theta$ is positive on each fiber, the matrix $(\phi_{\lambda\bar{\mu}})$ is positive definite, and we denote its inverse by $(\phi^{\lambda\bar{\mu}})$. The horizontal component $\Theta_\mathcal{H}$ and the vertical component $\Theta_\mathcal{V}$ have the local expressions 
\begin{align}
 &\Theta_\mathcal{H}=\sum_{i,j}( \phi_{i\bar{j}}-\sum_{\lambda,\mu} \phi_{i\bar{\mu}}\phi^{\lambda \bar{\mu}}\phi_{\lambda \bar{j}})dt_i\wedge d\bar{t}_j\label{horizontal}\\  &\Theta_\mathcal{V}=\sum_{\lambda,\mu}\phi_{\lambda\bar{\mu}}\delta z_\lambda\wedge \delta \bar{z}_{\mu}\label{vertical}
\end{align}
where $\delta z_\lambda=dz_\lambda+\sum_{i,\mu}\phi^{\lambda\bar{\mu}}\phi_{i\bar{\mu}}dt_i$. That (\ref{horizontal}) and (\ref{vertical}) are independent of local coordinates is discussed in
\cite[Subsection 1.1]{FLWgeodesiceinstein}, and the horizontal component $\Theta_\mathcal{H}$ is called the 
geodesic curvature there. The next two lemmas regarding the horizontal component $\Theta_\mathcal{H}$ will play an essential role in the proof of Lemma \ref{lem}.

We consider $\beta\in C^{\infty}(\mathcal{X}, p^*(\wedge^{1,1}T^*Y))$ and call $\beta$ positive if the matrix $(\beta_{i\bar{j}})$ in the local expression  $i\sum_{i,j}\beta_{i\bar{j}}dt_i\wedge d\bar{t}_j$ is positive. The following fact is implicitly stated in \cite[Lemma 6]{wu2025uniform}.     
\begin{lemma}\label{lem wzw}
For a positive $\beta\in C^{\infty}(\mathcal{X}, p^*(\wedge^{1,1}T^*Y))$, we have $\Theta_\mathcal{H}\wedge \beta^{m-1}>0$ if and only if  $\Theta^{n+1}\wedge \beta^{m-1}>0$.  \end{lemma}
\begin{proof}

Using local coordinates, we have a formula:
\begin{equation}\label{wzw formula}
\begin{aligned}
\Theta^{n+1}\wedge\beta^{m-1}
  =(n+1)! (m-1)!
 \sum_{i,j} \beta^{i\bar{j}}(\phi_{i\bar{j}}-\sum_{\lambda,\mu} \phi_{i\bar{\mu}}\phi^{\lambda \bar{\mu}}\phi_{\lambda \bar{j}})\det (\phi_{\lambda\bar{\mu}})
   \det(\beta)\\
  \big(\bigwedge^m_{k=1} i dt_k\wedge d\Bar{t}_k\wedge \bigwedge^n_{\lambda=1} i dz_\lambda\wedge d\Bar{z}_\lambda\big). \end{aligned}
\end{equation}
For the proof of formula (\ref{wzw formula}), see line -10 to line -1 on page 9 in \cite{wu2025uniform}. Meanwhile, by (\ref{horizontal}),
\begin{equation}\label{horizontal'}
 \Theta_\mathcal{H}\wedge \beta^{m-1}=\sum_{i,j} \beta^{i\bar{j}}(\phi_{i\bar{j}}-\sum_{\lambda,\mu} \phi_{i\bar{\mu}}\phi^{\lambda \bar{\mu}}\phi_{\lambda \bar{j}}) \beta^m/m. 
\end{equation}
From (\ref{wzw formula}) and (\ref{horizontal'}), we see that the signs of $\Theta^{n+1}\wedge \beta^{m-1}$ and $\Theta_\mathcal{H}\wedge \beta^{m-1}$
both depend on $\sum_{i,j} \beta^{i\bar{j}}(\phi_{i\bar{j}}-\sum_{\lambda,\mu} \phi_{i\bar{\mu}}\phi^{\lambda \bar{\mu}}\phi_{\lambda \bar{j}})$, so the lemma follows.
\end{proof}

Next, we assume the curvature $\Theta$ of $h$ is uniformly weakly RC-positive in the sense of Definition \ref{def 2}: $\Theta$ is positive on every fiber $\mathcal{X}_t$, and for any point $t\in Y$, there exists a nonzero $v\in T_tY$ such that $\Theta(\Tilde{v},\bar{\Tilde{v}})|_{(t,z)}>0$ for any lift $\Tilde{v}$ of $v$ to $T_{(t,z)}\mathcal{X}$. We have the following fact from \cite[Lemma 5]{wu2025uniform}. 
\begin{lemma}\label{lemma assumption}
 $\Theta(\Tilde{v},\bar{\Tilde{v}})|_{(t,z)}>0$ for any lift $\Tilde{v}$ of $v$ to $T_{(t,z)}\mathcal{X}$ if and only if  $\Theta_{\mathcal{H}}(\tilde{v},\bar{\Tilde{v}})>0$ on $\mathcal{X}_t$.   
\end{lemma}
By formula (\ref{horizontal}), we see that $\Theta_{\mathcal{H}}(\tilde{v},\bar{\Tilde{v}})$ is independent of the choice of lifts $\tilde{v}$, so we will simply write $\Theta_{\mathcal{H}}(v,\bar{v})$ from now on.

\section{Proofs of main results}\label{sec main}

We start with the proof of Lemma \ref{lem}, which is inspired by \cite[Proposition 3.6]{LiZhangZhang}.
\begin{proof}[Proof of Lemma \ref{lem}]
    We first fix a Hermitian metric $\gamma$ on $Y$. For $t\in Y$, $z\in \mathcal{X}_t$, and $v\in T_tY$ with $\gamma(v,\bar{v})=1$, we consider a continuous map $(t,z,v)\mapsto \Theta_\mathcal{H}(v,\bar{v})|_{(t,z)}$. For a fixed $t\in Y$, we consider the infimum of $ \Theta_\mathcal{H}(v,\bar{v})|_{(t,z)}$ over $z\in \mathcal{X}_t$ and $v\in T_tY$ with $\gamma(v,\bar{v})=1$, and we denote this infimum  by $c(t)$. 
    
    According to Lemma \ref{lemma assumption}, the assumption that $\Theta$ is uniformly weakly RC-positive implies: for $t\in Y$, there exists $v(t)\in T_tY$ with $\gamma(v(t),\overline{v(t)})=1$ such that 
    \begin{equation}\label{uni}
     \Theta_\mathcal{H}(v(t),\overline{v(t)})>0  \text{ on } \mathcal{X}_t.
    \end{equation}
    By continuity, the infimum of $\Theta_\mathcal{H}(v(t),\overline{v(t)})$ over $\mathcal{X}_t$ is realized at some point in $\mathcal{X}_t$, and so this infimum is positive. We denote $\inf_{\mathcal{X}_t}\Theta_\mathcal{H}(v(t),\overline{v(t)})$ by $\mu(t)$, which is positive. 
    
    Now, for a fixed $t_0\in Y$, by the previous paragraph, there exists $v(t_0)\in T_{t_0}Y$ with the property (\ref{uni}). Let $\{e_j\}_{j=1}^m$ be a local smooth frame of $TY$ around $t_0$ orthonormal with respect to $\gamma$ such that $e_1(t_0)=v(t_0)$. Let $\{\theta_j\}$ be the frame dual to $\{e_j\}$. Choose a positive number $a_0$  such that \begin{equation}\label{a0}
    (m-1)c(t_0)a_0>-\mu(t_0)/2,\end{equation}
this is doable since $\mu(t_0)>0$. Define a local Hermitian metric $\alpha_0=\theta_1\wedge \bar{\theta}_1+a_0^{-1}\sum_{j=2}^m \theta_j\wedge \bar{\theta}_j$. If we write $\Theta_\mathcal{H}=\sum_{j,k}(\Theta_\mathcal{H})_{j\bar{k}}\theta_j\wedge\bar{\theta}_k$, then for any point on the fiber $ \mathcal{X}_{t_0}$, 
    \begin{align*}
        &\Theta_\mathcal{H}\wedge p^*\alpha_0^{m-1}
        =\sum_{j,k} \alpha_0^{j\bar{k}}  (\Theta_\mathcal{H})_{j\bar{k}} \frac{p^*\alpha_0^m}{m}\\    
        =&\big[ \Theta_\mathcal{H}(e_1(t_0), \overline{e_1(t_0)})  +a_0\sum_{j=2}^m \Theta_\mathcal{H}(e_j(t_0), \overline{e_j(t_0)})    \big]\frac{p^*\alpha_0^m}{m}\\
       \geq& \big[ \Theta_\mathcal{H}(v(t_0), \overline{v(t_0)}) +a_0 (m-1)c(t_0)   \big]\frac{p^*\alpha_0^m}{m} \geq \big[  \mu(t_0)+  a_0 (m-1)c(t_0)    \big]\frac{p^*\alpha_0^m}{m}. \end{align*}
Therefore, for any point on the fiber  $ \mathcal{X}_{t_0}$, we use (\ref{a0}) to get
\begin{equation}\label{local}
  \Theta_\mathcal{H}\wedge p^*\alpha_0^{m-1}> \frac{\mu(t_0)}{2}\frac{p^*\alpha_0^m}{m}=\frac{\mu(t_0)}{2}  \frac{a_0^{-(m-1)}p^*\gamma^m}{m}.  
\end{equation}
By continuity, there exists a neighborhood $B(t_0)$ of $t_0$ in $Y$, such that $\Theta_\mathcal{H}\wedge p^*\alpha_0^{m-1}> \mu(t_0)  a_0^{-(m-1)}p^*\gamma^m/2m$ in $p^{-1}(B(t_0))$. By second-countability of $Y$, there exist countably many points $\{t_j\}_{j=1}^\infty$ in $Y$ each of which has a neighborhood $B_j$ with a locally defined Hermitian metric $\alpha_j$ and a positive number $a_j$ such that 
\begin{equation}
    \Theta_\mathcal{H}\wedge p^*\alpha_j^{m-1}>\frac{\mu(t_j)}{2}  \frac{a_j^{-(m-1)}p^*\gamma^m}{m} \text{ in } p^{-1}(B_j).
\end{equation} 
Let $\{f_j\}$ be a partition of unity subordinate to the countable open cover $\{B_j\}$, and define  $\eta=\sum_{j=1}^\infty f_j \alpha_j^{m-1}$, which is a positive $(m-1,m-1)$-form. By \cite[Page 279]{michelsohn}, there exists a unique Hermitian metric $\alpha$ on $Y$ such that $\alpha^{m-1}=\eta$. As a consequence,
\begin{equation}
    \Theta_\mathcal{H}\wedge p^*\alpha^{m-1}=\sum_j (f_j\circ p) \Theta_\mathcal{H}\wedge p^*\alpha_j^{m-1}\geq \sum_j (f_j\circ p)  \frac{\mu(t_j)a_j^{-(m-1)}}{2m}p^*\gamma^m>0.
\end{equation}
So, $\Theta^{n+1}\wedge p^*\alpha^{m-1}>0$ by Lemma \ref{lem wzw} and the proof is done.
\end{proof}

We prove Theorem \ref{thm general} here.
\begin{proof}[Proof of Theorem \ref{thm general}]
By Lemma \ref{lem}, there exists a Hermitian metric $\alpha$ on $X$ such that $\Theta^{r}\wedge p^*\alpha^{n-1}>0$ where $\Theta$ is the curvature of the metric $h$ on $O_{P(E^*)}(1)$. Note that the dimension of $X$ is $n$ and the dimension of the fiber $P(E^*_t)$ is $r-1$.

In order to use Theorem \ref{thm mean}, 
we choose $O_{P(E^*)}(k)\otimes K^{-1}_{P(E^*)/X}$ for $L$ with the metric $h^k\otimes g$, where $g$ is an arbitrary metric on $K^{-1}_{P(E^*)/X}$. The curvature of the metric $h^k\otimes g$ is $k\Theta+\Theta_g$,
where we denote by $\Theta_g$ the curvature of $g$. The associated bundle $V$ is $S^kE$. Since we know that $\Theta$ is positive on every fiber $P(E^*_t)$ and $\Theta^{r}\wedge p^*\alpha^{n-1}>0$, the assumption of Theorem \ref{thm mean} is fulfilled after taking $k$ large; that is, 
\begin{align*}
(k\Theta+\Theta_g)|_  {P(E^*_t)}>0 \text{ for any $t\in X$ and } (k\Theta+\Theta_g)^{r}\wedge p^*\alpha^{n-1}>0.    
\end{align*}
 Hence, the Hermitian bundle  $(S^kE, H_k)$ has positive mean curvature $\Lambda_\alpha \Theta^{H_k}>0$ for $k$ large where $H_k$ is the $L^2$-Hermitian metric in (\ref{L2}). 

By \cite[Theorem 1.4]{LiZhangZhang}, the positivity of $\Lambda_\alpha \Theta^{H_k}$ implies $\mu_L(S^kE,\alpha_G)>0$ where $\alpha_G$ is the Gauduchon metric conformal to $\alpha$ and $\mu_L$ is the minimum Harder--Narasimhan slope. But $\mu_L(S^kE,\alpha_G)=k\mu_L(E,\alpha_G)$, so $\mu_L(E,\alpha_G)>0$. By \cite[Theorem 1.4]{LiZhangZhang} again, there exists a Hermitian metric $H$ on $E$ such that $\Lambda_{\alpha_G} \Theta^H>0$. This completes the proof.

Here, we give another proof to Theorem \ref{thm general}. By \cite[Theorem 3]{wu2025uniform}, since $E$ is uniformly weakly RC-positive, $S^kE$ admits a Hermitian metric $H_k$, which is uniformly RC-positive for $k$ large. By \cite[Proposition 3.6]{LiZhangZhang}, there exists a Hermitian metric $\alpha$ on $X$ such that $\Lambda_\alpha \Theta^{H_k}>0$. Then, following the same argument as in the previous paragraph, we are done.\end{proof}
Finally, we give the proof of Theorem \ref{thm1}.
\begin{proof}[Proof of Theorem \ref{thm1}]
According to \cite[Corollary 1.5]{YangCamb}, if the holomorphic tangent bundle $TX$ of a compact K\"ahler manifold $X$ admits a Hermitian metric $H$ with positive mean curvature $\Lambda_\alpha \Theta^H>0$ with respect to some Hermitian metric $\alpha$ on $X$, then $X$ is projective and rationally connected. By Theorem \ref{thm general}, Theorem \ref{thm1} follows immediately.
\end{proof}

\section{Quasi-positivity}\label{sec quasi}

We start with the general fibration $p:\mathcal{X}^{m+n}\to Y^m$ and the Hermitian line bundle $(L,h)\to \mathcal{X}$ from the Introduction. Theorem \ref{thm quantitative} is proved by refining the argument in \cite[Theorem 2]{wu2025mean}. 
\begin{proof}[Proof of Theorem \ref{thm quantitative}]
Our goal is to show that, for any nonzero $u_0\in V_{t_0}$,
\begin{equation}\label{goal}
 H(\Lambda_\alpha \Theta^H u_0, u_0)>M\frac{m}{n+1} H(u_0,u_0).  
\end{equation}

First of all, we fix a coordinate system $(t_1,\dots, t_m)$ around $t_0$ in $Y$ such that the Hermitian metric $\alpha=i\sum_{j} dt_j\wedge d\bar{t}_j$ at $t_0$. By a standard argument, we extend $u_0$ to a local holomorphic section $u$ of $V$ such that $D' u =0$ at $t_0$ and $u(t_0)=u_0$, where $D'$ is the $(1,0)$-part of the Chern connection of the Hermitian bundle $(V,H)$. A straightforward computation gives
\begin{equation}\label{standard}
 \Lambda_\alpha  \partial \bar{\partial} H(u,u)=- H(\Lambda_\alpha \Theta^H u,u) \text{ at } t_0. 
\end{equation}
We then follow the same computation as in the proof of \cite[Theorem 2]{wu2025mean} and deduce \cite[Formula (17)]{wu2025mean}:
\begin{equation}\label{4.5'}
    \Lambda_\alpha\partial\bar{\partial}H(u,u)
 = 
    -c_n \Lambda_\alpha p_* ( u'\wedge\overline{u'}\wedge \partial\bar{\partial} \phi  e^{-\phi})
    +
    (-1)^n c_n \Lambda_\alpha p_*  (\bar{\partial}u'\wedge \overline{\bar{\partial}u'}e^{-\phi})\text{ at }  t_0,
\end{equation}
where $c_n=i^{n^2}$, $u'$ is an $(n,0)$-form representing $u$, and $\phi$ is a local weight of the metric $h$.
Due to the fact \begin{equation}\label{middle}
\begin{aligned}
       -c_n \Lambda_\alpha p_* ( u'\wedge\overline{u'}\wedge \partial\bar{\partial} \phi  e^{-\phi})\alpha^m/m&=-c_n  p_* ( u'\wedge\overline{u'}\wedge \partial\bar{\partial} \phi  e^{-\phi})\wedge \alpha^{m-1}\\&=-c_n  p_* ( u'\wedge\overline{u'}\wedge \partial\bar{\partial} \phi  e^{-\phi}\wedge p^*\alpha^{m-1}), 
\end{aligned}
\end{equation}
in order to estimate the middle term in (\ref{4.5'}), we will study
$ u'\wedge\overline{u'}\wedge \partial\bar{\partial} \phi  e^{-\phi}\wedge p^*\alpha^{m-1}$. 

For the $(n,0)$-form $u'$, we denote by $u_z dz$ the part in $u'$ with $dz_1\wedge\dots \wedge dz_n$, and by $u_{z_\lambda t_j}d\hat{z}_\lambda\wedge dt_j$ the part in $u'$ with $dz_1\wedge \dots \wedge dz_n\wedge dt_j$ omitting $dz_\lambda$ (note that $u_z$ and $u_{z_\lambda t_j}$ simply stand for coefficients not differentiation). Then, for any point on the fiber $\mathcal{X}_{t_0}$, the $(n+m,n+m)$-form $u'\wedge\overline{u'}\wedge \partial\bar{\partial} \phi  e^{-\phi} \wedge p^* \alpha^{m- 1} $ equals
\begin{align*}
e^{-\phi} \big(&\sum_{j,k} u_z dz\wedge \overline{u_zdz} \wedge\phi_{j\bar{k}} dt_j\wedge d\bar{t}_k \wedge p^*\alpha^{m-1}\\+&\sum_{\lambda,j,k}
u_z dz\wedge 
\overline{u_{z_\lambda t_k}d\hat{z}_\lambda\wedge dt_k}\wedge
\phi_{j\bar{\lambda}} dt_j\wedge d\bar{z}_\lambda \wedge p^*\alpha^{m-1}\\+&\sum_{\lambda,j,k}
u_{z_\lambda t_j}d\hat{z}_\lambda\wedge dt_j\wedge \overline{u_zdz}\wedge\phi_{\lambda \bar{k}} dz_\lambda \wedge d\bar{t}_k \wedge p^*\alpha^{m-1}
\\+
&\sum_{\lambda,\mu,j,k} u_{z_\lambda t_j}d\hat{z}_\lambda\wedge dt_j\wedge \overline{u_{z_\mu t_k}d\hat{z}_\mu\wedge dt_k}\wedge\phi_{\lambda \bar{\mu}} dz_\lambda\wedge d\bar{z}_\mu \wedge p^*\alpha^{m-1}\big),
\end{align*}
which can be further simplified as 
\begin{equation}\label{4.8}
\begin{aligned}
e^{-\phi}\big(\sum_j|u_z|^2\phi_{j\bar{j}}+&\sum_{\lambda, j} (-1)^{n-\lambda+1}u_z \overline{u_{z_\lambda t_j}} \phi_{j\bar{\lambda}}+\sum_{\lambda, j} (-1)^{n-\lambda+1}\overline{u_z} u_{z_\lambda t_j}\phi_{\lambda\bar{j}}\\+&\sum_{\lambda,\mu,j} (-1)^{\lambda+\mu}u_{z_\lambda t_j} \overline{u_{z_\mu t_j}}\phi_{\lambda \bar{\mu}} \big)
dz\wedge d\bar{z}\wedge \frac{p^*\alpha^m}{m};
\end{aligned}    
\end{equation}
in the above simplification, we use the fact $\alpha=i\sum_{j} dt_j\wedge d\bar{t}_j$ at $t_0$. For fixed $j$, if we denote the matrix $(\phi_{\lambda\bar{\mu}})$ by $A$, the column vector $(\phi_{\lambda\bar{j}})$ by $\phi_j$, and the column vector $((-1)^{n-\lambda+1}u_{z_\lambda t_j})$ by $B_j$, then after completing the square, (\ref{4.8}) equals
\begin{equation}\label{4.9}
e^{-\phi}\sum_j\big(|u_z|^2\phi_{j\bar{j}}-\sum_{\lambda,\mu}\phi_{\lambda\bar{j}}\phi^{\lambda\bar{\mu}}\phi_{j\bar{\mu}}|u_z|^2+\|\sqrt{A^{-1}}\phi_j\overline{u_z}+\sqrt{A}\overline{B_j}\|^2 \big)dz\wedge d\bar{z}\wedge \frac{p^*\alpha^m}{m}. 
\end{equation}

On the other hand, the assumption in Theorem \ref{thm quantitative}, $\Theta^{n+1}\wedge p^*\alpha^{m-1}>M \Theta^{n}\wedge p^*\alpha^{m}$ for any point on $\mathcal{X}_{t_0}$, has the following local expression.
\begin{equation}\label{assumption}
\begin{aligned}
(n+1)! (m-1)!
 \sum_{i,j} \alpha^{i\bar{j}}(\phi_{i\bar{j}}-\sum_{\lambda,\mu} \phi_{i\bar{\mu}}\phi^{\lambda \bar{\mu}}\phi_{\lambda \bar{j}})\det (\phi_{\lambda\bar{\mu}})
   \det(\alpha)  \big(\bigwedge^m_{k=1} i dt_k\wedge d\Bar{t}_k\wedge \bigwedge^n_{\lambda=1} i dz_\lambda\wedge d\Bar{z}_\lambda\big)\\
>M n!m! \det(\phi_{\lambda\bar{\mu}})\det(\alpha)\big(\bigwedge^m_{k=1} i dt_k\wedge d\Bar{t}_k\wedge \bigwedge^n_{\lambda=1} i dz_\lambda\wedge d\Bar{z}_\lambda\big).   
\end{aligned}
\end{equation}
This local expression is deduced from (\ref{wzw formula}). 

Combining (\ref{4.8}), (\ref{4.9}), and (\ref{assumption}), we obtain
\begin{equation}\label{ineq}
u'\wedge\overline{u'}\wedge \partial\bar{\partial} \phi  e^{-\phi} \wedge p^* \alpha^{m- 1}>M\frac{m}{n+1} e^{-\phi} |u_z|^2 dz\wedge d\bar{z}\wedge \frac{p^*\alpha^m}{m} \text{ for any point on } \mathcal{X}_{t_0}.
\end{equation}
Integrating (\ref{ineq}) along $\mathcal{X}_{t_0}$ and multiplying by $-c_n$, we get 
\begin{equation}\label{long}
\begin{aligned}
-c_n \int_{\mathcal{X}_{t_0}}u'\wedge\overline{u'}\wedge \partial\bar{\partial} \phi  e^{-\phi} \wedge p^* \alpha^{m- 1}&<M\frac{m}{n+1}\big( -c_n \int_{\mathcal{X}_{t_0}}e^{-\phi} |u_z|^2 dz\wedge d\bar{z}\wedge \frac{p^*\alpha^m}{m}\big)\\&=-M\frac{m}{n+1}H(u,u)\frac{\alpha^m}{m},
\end{aligned}
\end{equation}
where the equality is due to the definition of the $L^2$-metric $H$ in (\ref{L2}), that is, $H(u,u)=p_*(c_n u'\wedge \overline{u'}e^{-\phi})$. Using (\ref{middle}) and (\ref{long}), the middle term in (\ref{4.5'}) satisfies
\begin{equation}\label{M} -c_n \Lambda_\alpha p_* ( u'\wedge\overline{u'}\wedge \partial\bar{\partial} \phi  e^{-\phi})<
    -M\frac{m}{n+1}H(u,u) \text{ at } t_0.
\end{equation}
Since the last term in (\ref{4.5'}) is nonpositive by the argument in \cite[Formula (18)]{wu2025mean}, we obtain from (\ref{4.5'}) and (\ref{M}) that 
$\Lambda_\alpha\partial\bar{\partial}H(u,u)<-M\frac{m}{n+1}H(u,u)
$ at $t_0$, and so 
$H(\Lambda_\alpha \Theta^H u,u)>    M\frac{m}{n+1}H(u,u)$ at $t_0$ by (\ref{standard}). 
\end{proof}
We prove Theorem \ref{thm quasi to full} here.
\begin{proof}[Proof of Theorem \ref{thm quasi to full}]
We apply Theorem \ref{thm quantitative} to the fibration $P(E^*)\to X$, and
we choose $O_{P(E^*)}(k)\otimes K^{-1}_{P(E^*)/X}$ for the line bundle $L$, so $V$ is $S^kE$.

We need to equip the line bundle $O_{P(E^*)}(k)\otimes K^{-1}_{P(E^*)/X}$ with a suitable metric. Instead of using an arbitrary metric on $ K^{-1}_{P(E^*)/X}$, the following choice will be more effective:
due to the isomorphism
\begin{equation*}
 K^{-1}_{P(E^*)/X}\simeq  O_{P(E^*)}(r)\otimes p^*(\det E)^{-1}, 
\end{equation*}
if $g$ is an arbitrary metric on $(\det E)^{-1}$, then we have the metric  $h^r\otimes p^*g$ on $ K^{-1}_{P(E^*)/X}$. Therefore, the line bundle $O_{P(E^*)}(k)\otimes K^{-1}_{P(E^*)/X}$ is equipped with the metric $h^{k+r}\otimes p^*g$ whose curvature is $(k+r)\Theta+p^*\Theta_g$, where we denote by $\Theta_g$ the curvature of $g$. Moreover, by the binomial expansion and a degree count, we have
\begin{align}
\big((k+r)\Theta+p^*\Theta_g\big)^r\wedge p^*\alpha^{n-1}&=(k+r)^r\Theta^r\wedge p^*\alpha^{n-1}+r(k+r)^{r-1}\Theta^{r-1}\wedge p^*\Theta_g\wedge p^*\alpha^{n-1},\label{binomial}\\
\big((k+r)\Theta+p^*\Theta_g\big)^{r-1}\wedge p^*\alpha^{n}&=(k+r)^{r-1}\Theta^{r-1}\wedge p^*\alpha^n.\label{binomial2}
\end{align}
By compactness of $P(E^*)$, the last term in (\ref{binomial}) has the following rough lower bound 
\begin{equation}\label{rough bound}
\Theta^{r-1}\wedge p^*\Theta_g\wedge p^*\alpha^{n-1}> -C \Theta^{r-1}\wedge p^*\alpha^{n} \text{ on }P(E^*)    
\end{equation}
for some positive constant $C$ independent of $k$. Therefore, we deduce from (\ref{binomial}) that
\begin{equation}\label{global}
 \big((k+r)\Theta+p^*\Theta_g\big)^r\wedge p^*\alpha^{n-1}>-Cr(k+r)^{r-1} \Theta^{r-1}\wedge p^*\alpha^{n}\text{ on } P(E^*)    
\end{equation} because $\Theta^r\wedge p^*\alpha^{n-1}\geq 0$ on $P(E^*)$ by the assumption of Theorem \ref{thm quasi to full}. Using (\ref{binomial2}), inequality (\ref{global}) can be written as
\begin{equation}\label{global'}
  \big((k+r)\Theta+p^*\Theta_g\big)^r\wedge p^*\alpha^{n-1}>-Cr\big((k+r)\Theta+p^*\Theta_g\big)^{r-1}\wedge p^*\alpha^{n} \text{ on } P(E^*).  
\end{equation}

By the assumption of Theorem \ref{thm quasi to full}, there exists $t_0\in X$ such that
$\Theta^r\wedge p^*\alpha^{n-1}>0$ for any point on the fiber $P(E^*_{t_0})$. By continuity, there exists a neighborhood $B$ of $t_0$ in $X$ such that
\begin{equation}\label{lower bound} \Theta^r\wedge p^*\alpha^{n-1}>D \Theta^{r-1}\wedge p^*\alpha^{n} \text{ in } p^{-1}(B)   
\end{equation}
for some positive constant $D$. The constant $D$ and the neighborhood $B$ are both independent of $k$. Using (\ref{rough bound}) and (\ref{lower bound}), we deduce from (\ref{binomial}) that, in $p^{-1}(B)$, 
\begin{equation}\label{local'}
\begin{aligned}
\big((k+r)\Theta+p^*\Theta_g\big)^r\wedge p^*\alpha^{n-1}>& \big((k+r)^rD-r(k+r)^{r-1}C\big) \Theta^{r-1}\wedge p^*\alpha^{n}\\
=&\big((k+r)D-rC\big)\big((k+r)\Theta+p^*\Theta_g\big)^{r-1}\wedge p^*\alpha^{n},
\end{aligned}
\end{equation}
where we use (\ref{binomial2}) in the equality. Since the curvature restricted to any fiber $\big((k+r)\Theta+p^*\Theta_g\big)|_{P(E^*_t)}=(k+r)\Theta|_{P(E^*_t)}$ is positive, we can apply Theorem \ref{thm quantitative} to (\ref{global'}) and (\ref{local'}). If we denote by $H_k$ the corresponding $L^2$-metric on $S^kE$, then the curvature $\Theta^{H_k}$ of the Hermitian metric $H_k$ satisfies 
\begin{equation}\label{two ineq}
\begin{aligned}
&\Lambda_\alpha \Theta^{H_k}>-rC\frac{n}{r}\Id_{S^kE} \text{ in } X,\\
&\Lambda_\alpha \Theta^{H_k}>\big((k+r)D-rC\big) \frac{n}{r}\Id_{S^kE}\text{ in } B.   
\end{aligned}
\end{equation}

Let $\tilde{\alpha}$ be the Gauduchon metric conformal to $\alpha$. So, $\tilde{\alpha}=e^f\alpha$ for some smooth function $f$ on $X$, and $\Lambda_{\tilde{\alpha}}\Theta^{H_k}=e^{-f}\Lambda_\alpha\Theta^{H_k}$. We then choose $k$ so large that
\begin{equation}\label{inequality}
\int_{B}e^{-f}\big((k+r)D-rC\big)\frac{n}{r}\tilde{\alpha}^n- \int_{X\setminus B} e^{-f}rC\frac{n}{r} \tilde{\alpha}^n>0;  
\end{equation}
this is possible because the dominant term $(k+r)$ is in the first integral. If we denote the smallest eigenvalue of $\Lambda_{\tilde{\alpha}}\Theta^{H_k}$ by $\lambda_{\min}$, then according to (\ref{two ineq}),
\begin{equation}
  \int_X \lambda_{\min}\tilde{\alpha}^n>\int_{B}e^{-f}\big((k+r)D-rC\big)\frac{n}{r}\tilde{\alpha}^n- \int_{X\setminus B} e^{-f}rC\frac{n}{r} \tilde{\alpha}^n,
\end{equation}
which is positive by (\ref{inequality}). Then by \cite[Remark 1.6]{LiZhangZhang}, there exists a Hermitian metric $H'_k$ on $S^kE$, probably different from $H_k$, such that $\Lambda_{\tilde{\alpha}}\Theta^{H'_k}>0$.  

By \cite[Theorem 1.4]{LiZhangZhang}, the fact $\Lambda_{\tilde{\alpha}}\Theta^{H'_k}>0$ implies $\mu_L(S^kE,\tilde{\alpha})>0$. But $\mu_L(S^kE,\tilde{\alpha})=k\mu_L(E,\tilde{\alpha})$, so $\mu_L(E,\tilde{\alpha})>0$. By \cite[Theorem 1.4]{LiZhangZhang} again, there exists a Hermitian metric $H$ on $E$ such that $\Lambda_{\tilde{\alpha}} \Theta^H>0$. Since $\Lambda_{\tilde{\alpha}} \Theta^H=e^{-f}\Lambda_\alpha\Theta^H$, we have $\Lambda_\alpha\Theta^H>0$.
\end{proof}

Finally, we present the proof of Theorem \ref{thm uniruled}.
\begin{proof}[Proof of Theorem \ref{thm uniruled}]
We first consider the general fibration $p:\mathcal{X}^{n+m}\to Y^m$ and a Hermitian line bundle $(L,h)$ over $\mathcal{X}$ as in the Introduction. Assume $h$ is uniformly weakly RC-quasi-positive. That is, it is
uniformly weakly RC-semipositive everywhere and uniformly weakly RC-positive somewhere (i.e., there exist a point $t\in Y$ and a nonzero $v\in T_tY$ such that $\Theta(\tilde{v},\bar{\tilde{v}})|_{(t,z)}>0$ for any lift $\tilde{v}$ of $v$ to $T_{(t,z)}\mathcal{X}$). By \cite[Theorem 4]{wu2025uniform}, the Hermitian bundle $(V,H)$ is uniformly RC-quasi-positive (\cite[Theorem 4]{wu2025uniform} covers only the positive case, but the semipositive case can be deduced similarly).

Now, let us apply this result to the setup of Theorem \ref{thm uniruled}, $P(E^*)\to X$ with $E=TX$. We use the weaker assumption that $E$ is uniformly weakly RC-quasi-positive, namely, there exists a metric $h$ on $O_{P(E^*)} (1)$
 with the quasi-positivity property. We choose $O_{P(E^*)}(r)$ for $L$ with the metric $h^r$, so the bundle $V$ is $\det E$, and we denote the $L^2$-metric on $\det E$ by $H$. According to the result in the first paragraph, the Hermitian bundle ($\det E,H$) is uniformly RC-quasi-positive, hence RC-quasi-positive. This completes the proof because the line bundle $\det E=\det TX=K^{-1}_X$. 
\end{proof}

\bibliographystyle{amsalpha}
\bibliography{Dominion}

\textsc{National Central University, No.300, Zhongda Road, Taoyuan, Taiwan}

\texttt{\textbf{wuuuruuu@gmail.com}}

\end{document}